\documentclass[11pt]{article}
\usepackage[a4paper,margin=1in]{geometry}
\usepackage{amsmath,amssymb,amsthm,mathtools,bm}
\usepackage{enumitem}
\usepackage{microtype}
\usepackage{hyperref}

\newtheorem{theorem}{Theorem}

\newtheorem{lemma}{Lemma}

\newcommand{\R}{\mathbb{R}}

\newcommand{\C}{\mathbb{C}}
\newcommand{\abs}[1]{\lvert #1\rvert}

\newcommand{\supp}{\mathrm{spt}}
\newcommand{\ceq}{:=}
\DeclareMathOperator*{\Res}{Res}
\begin{document}
	
	\title{Solving Stengle’s Example in Rational Arithmetic:\\ Exact Values of the Moment-SOS Relaxations}
	\author{Didier Henrion}
	\date{\today}
	\maketitle
	
	\begin{abstract}
		We revisit Stengle's classical univariate polynomial optimization example $\min 1-x^2$ s.t. $(1-x^2)^3\ge 0$ whose constraint description is degenerate at the minimizers.
		We prove that the moment-SOS hierarchy of relaxation order $r\ge 3$ has the exact value $-1/r(r-2)$. For this we construct in rational arithmetic a dual polynomial sum-of-squares (SOS) certificate and a primal moment sequence representing a finitely atomic measure. The key ingredients are elementary trigonometric properties of Chebyshev and Gegenbauer polynomial, and a Christoffel-Darboux kernel argument.
	\end{abstract}

	\section{Introduction}
	
	Given a polynomial optimization problem (POP), the moment-SOS hierarchy (aka the Lasserre hierarchy \cite{Lasserre2001}) builds a sequence of semidefinite relaxations of increasing size
	indexed by a relaxation order $r$.
	On the primal side, one optimizes a linear functional over truncated moment sequences subject
	to semidefinite constraints on the moment matrix and localizing matrices. On the dual side, one searches for
	a sum-of-squares (SOS) decomposition of a shifted polynomial.
	See \cite{Henrion2020,Nie2023,Theobald2024} for recent general overviews.
	
	Recent work has developed \emph{worst-case upper bounds} on the relaxation error:
	for broad classes of instances one can guarantee that the gap after $r$ steps is at most
	$O(r^{-k})$, with the exponent $k$ depending on the geometry of the domain
	and regularity assumptions; see \cite{Slot2022,Laurent2024,Gribling2025,Schlosser2026} and references
	therein.
	Such results are \emph{sufficiency} statements: they provide a safety guarantee that the hierarchy cannot
	converge slower than a stated rate, even on the hardest instances in the class.
	
	In contrast, \emph{lower bounds} on the relaxation error show that certain upper bounds are essentially sharp:
	one constructs explicit instances where the hierarchy converges no faster than a given rate. These hardness results are
	\emph{necessity} statements: they quantify intrinsic limitations of the method (or of the representation
	class), and rules out uniformly better guarantees without changing assumptions or algorithms. The earliest quantitative lower bound predates the modern formulation of the moment-SOS hierarchy and is due to Stengle \cite{Stengle1996}; see the discussions in \cite[Chapter~17]{Schmuedgen2017} and \cite[Section~3.5.1]{Nie2023}.
	Stengle provided a univariate example showing a lower bound of $\Omega(1/r^2)$ and an upper bound of $O((\log r)^2/r^2)$. This example is characterized by a degenerate description of the domain.
	Recently, the paper \cite{Baldi2024} gives lower bounds for domains described by non-degenerate inequalities, connecting these with quantitative non-stability phenomena \cite{Scheiderer2005}.
	Lower bounds were also constructed for a specific parametric univariate POP in \cite{Henrion2025}, but in this case the moment-SOS hierarchy has always finite (i.e. non-asymptotic) convergence, contrary to the Stengle example.

	\paragraph{Contribution.}
	The Stengle example appears as Example~1 in \cite{Kocvara2025},
	where high-precision semidefinite programming experiments suggest a $\Theta(1/r^2)$ behavior. We solve Stengle's example analytically, in rational arithmetic, and prove that the relaxation error is \emph{exactly} $-1/r(r-2)$.
	The proof is based on convex duality and basic properties of orthogonal polynomials (Chebyshev and Gegenbauer polynomials, and the
	Christoffel-Darboux kernel) and elementary trigonometric identities, mirroring the role of orthogonal
	polynomials in upper-bound analyses \cite{Slot2022,Laurent2024,Schlosser2026}.
	
	\section{Stengle's example and its moment-SOS relaxations}
	
	\subsection{The POP and its degeneracy}
	
	Consider the univariate constrained polynomial optimization problem (POP)
	\begin{equation}\label{eq:POP}\tag{POP}
		\min_{x\in\R}\ f(x)\ceq 1-x^2
		\qquad\text{s.t.}\qquad
		g(x)\ceq (1-x^2)^3\ \ge 0.
	\end{equation}
	Since $g(x)\ge 0$ holds if and only if $x\in[-1,1]$, the problem reduces to
	\[
	f^* \ceq \min_{|x|\le 1} (1-x^2)=0,
	\]
	attained at the boundary points $x^\star=\pm 1$.
	
	The POP is ill-posed because the feasible set is described by a single inequality $g(x)\ge 0$ whose gradient vanishes at the minimizers:
	\[
	g'(x) = -6x(1-x^2)^2,\qquad g'(\pm 1)=0.
	\]
	Thus the active constraint at $x^\star=\pm 1$ is \emph{degenerate}:
	the linearization carries no first-order information. The standard constraint qualifications of nonlinear programming are violated.
	
	For illustration, let us consider a classical   interior-point method consisting of minimizing, for a given barrier parameter $\mu>0$, a  log-barrier function
	\[
	h_\mu(x)\ceq f(x) - \mu \log g(x)
	\;=\;
	1-x^2 - 3\mu \log(1-x^2),
	\qquad x\in(-1,1).
	\]
	A stationary point satisfies
	\[
	\nabla h_\mu(x) = -2x + \frac{6\mu x}{1-x^2}=0,
	\]
	hence either $x=0$ or $1-x(\mu)^2 = 3\mu$ which yields
	\[
		x(\mu)=\sqrt{1-3\mu}\ \xrightarrow[\mu\downarrow 0]{}\ 1.
	\]
	The dual Lagrange multiplier induced by the log-barrier stationarity is
	\[
	\nabla f(x(\mu)) - \lambda(\mu)\nabla g(x(\mu))=0
	\quad\text{with}\quad
	\lambda(\mu)=\frac{\mu}{g(x(\mu))}=\frac{\mu}{(3\mu)^3}=\frac{1}{27\mu^2}\ \xrightarrow[\mu\downarrow 0]{}\ +\infty.
	\]
	Moreover, the curvature of the barrier objective is
	\[
	\nabla^2 h_\mu(x)= -2 + 6\mu\frac{1+x^2}{(1-x^2)^2},
	\]
	and substituting $x(\mu)^2=1-3\mu$ gives
	\[
		\nabla^2 h_\mu(x(\mu)) = -2 + 6\mu\frac{2-3\mu}{(3\mu)^2}
		= \frac{4}{3\mu}-4\ \xrightarrow[\mu\downarrow 0]{}\ +\infty.
	\]
	Thus the formulation \eqref{eq:POP} is a clean toy model for ill-conditioning of barrier methods:
	the primal iterates converge to the boundary while the dual multiplier and the barrier function Hessian blow up.
	
	\subsection{Moment and SOS relaxations}
	
	Fix an integer $r\ge 3$. Let $\R[x]_r$ denote the vector space of polynomials of the scalar indeterminate $x \in \R$ of degree at most $r$. In the monomial basis $b(x):=(1,x,\dots,x^r)^\top$ we identify a polynomial $p(x)$ with its coefficient vector $p \in \R^{r+1}$, i.e. $p(x) = p^\top b(x)$.
	A truncated moment sequence is a vector $y=(y_0,y_1,\dots,y_{2r})$.
	The associated Riesz functional $\ell_y$ acts linearly on $\R[x]_{2r}$ by
	$\ell_y(x^k)=y_k$ and extension by linearity.
	The moment matrix $M_r(y)$ of order $r$ is the Hankel matrix indexed by monomials $b(x)$:
	\begin{equation}\label{eq:mommat}	p^\top M_r(y)p = \ell_y(p(x)^2)\qquad\forall p\in\R[x]_r.
	\end{equation}
	The  localizing matrix for $g$ at order $r-3$ is
	\begin{equation}\label{eq:locmat}
	q^\top M_{r-3}(g\,y)\,q = \ell_y(g(x)\,q(x)^2)\qquad\forall q\in\R[x]_{r-3}.
	\end{equation}
	
	The order-$r$ moment relaxation of \eqref{eq:POP} is
		\begin{equation}\label{eq:MOM}\tag{MOM$_r$}
			\begin{aligned}
				\inf_y\ & \ell_y(f)\\
				\text{s.t. }& \ell_y(1)=1,\quad M_r(y)\succeq 0,\quad M_{r-3}(g\,y)\succeq 0 
			\end{aligned}
		\end{equation}
	and the order-$r$ SOS relaxation of \eqref{eq:POP} is
		\begin{equation}\label{eq:SOS}\tag{SOS$_r$}
			\begin{aligned}
				\sup_{\varepsilon,\,p,\,q}\ & \varepsilon\\
				\text{s.t. }&
				f(x)-\varepsilon = p(x) + g(x)\,q(x), \quad
				p\in\Sigma[x]_{2r},\quad q\in\Sigma[x]_{2(r-3)}
			\end{aligned}
		\end{equation}
	where the convex cone $\Sigma[x]_{2d} \subset \R[x]_{2d}$ consists of SOS polynomials of degree at most $2d$. The moment-SOS hierarchy for POP was introduced in this exact form in \cite{Lasserre2001}, see e.g. \cite{Henrion2020,Nie2023,Theobald2024} for recent overviews.
	
	\begin{lemma}[Weak duality]\label{prop:weak-duality}
		For every $r\ge 3$, any feasible $y$ in \eqref{eq:MOM} and any feasible
		$(\varepsilon,p,q)$ in \eqref{eq:SOS} satisfy $\ell_y(f)\ge \varepsilon$.
	\end{lemma}
	
	\begin{proof}
		If $f-\varepsilon=p+g q$ with $p,q$ SOS, then for any feasible $y$,
		\[
		\ell_y(f)-\varepsilon
		= \ell_y(p)+\ell_y(g q)
		\ge 0,
		\]
		because $\ell_y(p)\ge 0$ follows from $M_r(y)\succeq 0$ and
		$\ell_y(g q)\ge 0$ follows from $M_{r-3}(g y)\succeq 0$.
	\end{proof}
	
	\subsection{Strong duality and non-attainment at zero}
	
	The quadratic module generated by $g$ is the set of polynomials $p+gq$ for $p$ and $q$ SOS. We say that a quadratic module is Archimedean when it contains $R^2-x^2$ for some $R>0$. Note that this implies compactness of the set described by the inequality $g(x) \geq 0$. The convergence of the moment-SOS hierarchy, as originally proved in \cite{Lasserre2001}, relies on Putinar's Positivstellensatz (Psatz) \cite{Putinar1993}, which in turn relies on the Archimedean property of the quadratic module.
	
	The quadratic module generated by $g(x)=(1-x^2)^3$ is Archimedean since
		\begin{equation}\label{eq:archi}
			2-x^2 = \underbrace{(x^3-2x)^2 + (x^2-1)^2}_{p(x)} + (1-x^2)^3  \underbrace{(1)}_{q(x)}.
		\end{equation}
	Therefore Putinar's Psatz applies and the moment-SOS hierarchy converges, i.e.
		\[
		\lim_{r\to\infty}\varepsilon_r^\star = f^* =0.
		\]
 	Note however that for our POP the quadratic module and the preordering are the same, since we have only one generator $g$. Convergence of the moment-SOS hierarchy then also follows from Schmüdgen's Psatz \cite{Schmudgen1991} which applies without resorting to the Archimedean property, because the feasibility domain $[-1,1]$ is compact.
	
	\begin{lemma}[Strong duality and attainment]\label{prop:slater}
		For every $r\ge 3$, the primal \eqref{eq:MOM} satisfies Slater's condition, implying
		strong duality and dual attainment:
		\[
		\varepsilon_r^\star := \inf\eqref{eq:MOM} = \min\eqref{eq:MOM} = \sup\eqref{eq:SOS} = \max\eqref{eq:SOS}.
		\]
	\end{lemma}
	
	\begin{proof}
		Let $\mu$ be the normalized Lebesgue measure on $[-1,1]$ and let $y$ be its moment sequence
		($y_k=\int_{-1}^1 x^k\,d\mu$).
		Then for any nonzero $p\in\R[x]_{r}$, $\int p^2\,d\mu>0$, hence $M_r(y)\succ 0$.
		Similarly, $g(x)=(1-x^2)^3>0$ for all $x\in(-1,1)$, so for nonzero $q\in\R[x]_{r-3}$,
		$\int g q^2\,d\mu>0$ and $M_{r-3}(g y)\succ 0$.
		Thus \eqref{eq:MOM} is strictly feasible and conic duality yields strong duality and dual attainment.

		To prove primal attainment, observe that the feasible set is non-empty and closed. From identity \eqref{eq:archi}, for any feasible $y$ it holds $\ell_y(2-x^2)=\ell_y(p_0)+\ell_y(g)\ \ge\ 0.$
		Since $\ell_y(1)=1$, this gives $\ell_y(x^2)\le 2$. Next, for any integer $k\ge 2$ we have the
		pointwise inequality on $\R$: $x^{2k}\ \le\ 2^{k-1}x^2$ whenever  $x^2\le 2.$
		To use this within the moment constraints, we encode the bound $x^2\le 2$ via the SOS polynomial
		$(2-x^2)(x^{k-1})^2$, whose degree is $2k$.
		Indeed, for every $k\in\{1,\dots,r\}$ the polynomial $(2-x^2)\,x^{2(k-1)}$ is in the quadratic module. Hence for every feasible $y$, $0\ \le\ \ell_y\big((2-x^2)x^{2(k-1)}\big)
			= 2\,y_{2k-2}-y_{2k}.$
		This yields the recursion $y_{2k}\le 2y_{2k-2}$ for $k=1,\dots,r$.
		Since $y_0=1$, we obtain by induction $0\le y_{2k}\le 2^k$ for all $k=0,1,\dots,r$.
		For odd moments, the Cauchy-Schwarz inequality with $M_r(y)\succeq 0$ gives $|y_{2k+1}|^2 = |\ell_y(x^{k}\cdot x^{k+1})|^2 \le \ell_y(x^{2k})\,\ell_y(x^{2k+2}) = y_{2k}\,y_{2k+2}
		\le 2^{k}\,2^{k+1}=2^{2k+1}.$
		Thus all coordinates $y_0,\dots,y_{2r}$ are uniformly bounded over the feasible set. The objective $y\mapsto \ell_y(f)$ is linear hence continuous. Minimizing a continuous function
		over a nonempty compact set attains its minimum.
	\end{proof}
		
	\begin{lemma}\label{lem:eps-neg}
		Any dual-feasible triple $(\varepsilon,p,q)$ in \eqref{eq:SOS} satisfies $\varepsilon\le 0$.
		Moreover, $\varepsilon=0$ is \emph{not} feasible in \eqref{eq:SOS} for any finite $r$.
		In particular, $\varepsilon_r^\star<0$ for all $r\ge 3$.
	\end{lemma}
	
	\begin{proof}
		Evaluating $f(x)-\varepsilon=p(x)+g(x)q(x)$ at $x=1$ gives $-\varepsilon=p(1) \ge 0$.
		Assume by contradiction that $\varepsilon=0$ is feasible, i.e. 
		$1-x^2 = p(x) + (1-x^2)^3 q(x)$ for $p, q$ SOS.
		At $x=1$, the right-hand side equals $p(1)$, so $p(1)=0$.
		Since $p$ is SOS, every real root has even multiplicity; therefore $p$ vanishes at $x=1$ with
		multiplicity at least $2$.
		The term $(1-x^2)^3 q(x)$ vanishes at $x=1$ with multiplicity at least $3$.
		Hence the right-hand side vanishes at $x=1$ with multiplicity at least $2$,
		whereas the left-hand side $1-x^2$ has a simple zero at $x=1$, a contradiction.
	\end{proof}

	\subsection{Main theorem and proof strategy}
	
	\begin{theorem}[Exact relaxation value]\label{thm:main}
		For every integer $r\ge 3$, the order-$r$ moment--SOS relaxation value is
		\[
		\varepsilon_r^\star = -\frac{1}{r(r-2)}.
		\]
	\end{theorem}
	
	To prove this result, we will construct analytically:
	\begin{itemize}[leftmargin=2em]
		\item in Section~\ref{sec:SOS}, a dual-feasible SOS certificate  with $\varepsilon=-1/(r(r-2))$;
		\item in Section~\ref{sec:Moments}, a primal-feasible moment sequence $y$ with $\ell_y(1-x^2)=-1/(r(r-2))$.
	\end{itemize}
	The proof of Theorem \ref{thm:main} is then given in Section \ref{sec:proof}.
	
	\section{The SOS side: a pure-square identity}\label{sec:SOS}
	
	The numerical solution of \eqref{eq:SOS} suggests that optimal SOS polynomials are perfect squares.
	Motivated by this observation, we seek an identity of the form
	\begin{equation}\label{eq:identity}
		r(r-2)(1-x^2)+1 \;=\; A_r(x)^2 + 4\,(1-x^2)^3\,B_r(x)^2,
	\end{equation}
	for some $A_r \in \R[x]_r$, $B_r \in \R[x]_{r-3}$,
	which would yield the order-$r$ certificate
	\begin{equation}\label{eq:certificate}
	1-x^2+\frac{1}{r(r-2)}=\underbrace{\frac{A_r(x)^2}{r(r-2)}}_{p(x)} + (1-x^2)^3\underbrace{\frac{4B_r(x)^2}{r(r-2)}}_{q(x)},
	\end{equation}
	hence $\varepsilon=-1/(r(r-2))$ would be feasible in \eqref{eq:SOS}.
	
	\subsection{Chebyshev and Gegenbauer polynomials}
	
	Let $T_k$ denote the Chebyshev polynomial of the first kind, defined by the three-term recurrence
	\[
	T_{k+1}(x)=2xT_k(x)-T_{k-1}(x),\qquad T_0(x)=1,\ \ T_1(x)=x.
	\]
	See \cite[Ch.~22, \S22.7, Table~22.7.4]{AbramowitzStegun1964} for the recurrence and \cite[Ch.~22, \S22.4, Eq.~(22.4.4)]{AbramowitzStegun1964} for the initial values.
	
	Let $P_k^{(2)}$ denote the Gegenbauer polynomial of parameter $2$, defined by the recurrence
	\[
	(k+1)P_{k+1}^{(2)}(x)=2(k+2)x\,P_k^{(2)}(x)-(k+3)P_{k-1}^{(2)}(x),\qquad
	P_0^{(2)}(x)=1,\ \ P_1^{(2)}(x)=4x.
	\]
	Gegenbauer polynomials are also called Jacobi's ultraspherical polynomials \cite[Sec. 4.7]{Szego1975}, see \cite[(4.7.17)]{Szego1975} for the recurrence.
	See also \cite[Ch.~22, \S22.7, Eq.~(22.7.3)]{AbramowitzStegun1964} for the recurrence and \cite[Ch.~22, \S22.4, Eq.~(22.4.2)]{AbramowitzStegun1964} for the initial values.
	
	Define, for $r\ge 3$,
	\begin{equation}\label{eq:def-Ar-Br}
		A_r(x)\ceq \frac{r-2}{2}T_r(x)-\frac{r}{2}T_{r-2}(x),
		\qquad
		B_r(x)\ceq P_{r-3}^{(2)}(x).
	\end{equation}
	
	\subsection{Proof of the identity via a Pythagorean trigonometric argument}
	
	\begin{lemma}\label{lem:sos}
		For every $r\ge 3$, the polynomials $A_r,B_r$ defined in \eqref{eq:def-Ar-Br}
		satisfy the identity \eqref{eq:identity}.
	\end{lemma}
	
	\begin{proof}
	Fix $\theta\in\R$ and $x=\cos\theta$.
 	Recall the classical trigonometric representations
	\cite[Ch.~22, \S22.3, Eqs.~(22.3.15)--(22.3.16)]{AbramowitzStegun1964}, \cite[(1.12.3)]{Szego1975}:
	\begin{equation}\label{eq:trigo}
	T_k(\cos\theta)=\cos(k\theta),
	\qquad
	U_k(\cos\theta)=\frac{\sin(k+1)\theta}{\sin\theta}
	\end{equation}
	where $U_k$ denote the Chebyshev polynomial of the second kind, defined by the recurrence
\[
U_{k+1}(x)=2x\,U_k(x)-U_{k-1}(x),\qquad U_0(x)=1,\ \ U_1(x)=2x.
\]
See \cite[Ch.~22, \S22.7, Eq.~(22.7.5)]{AbramowitzStegun1964} for the recurrence  
and \cite[Ch.~22, \S22.4, Eq.~(22.4.5)]{AbramowitzStegun1964} for the initial values.  

	Moreover, the Chebyshev polynomial of the second kind is a special Gegenbauer polynomial:
	\[
	U_k(x)=P_k^{(1)}(x)
	\]
	see \cite[Ch.~22, \S22.5, Eq.~(22.5.34)]{AbramowitzStegun1964}) or \cite[(4.7.2)]{Szego1975}.
	Using the Gegenbauer derivative identity \cite[(4.7.14)]{Szego1975}:
	\[
	\frac{d}{dx}P_{k+1}^{(1)}(x)=2\,P_k^{(2)}(x),
	\]
	we obtain
	\[
	P_k^{(2)}(x)=\tfrac12\,\frac{d}{dx}U_{k+1}(x).
	\]
	Combining this with \eqref{eq:trigo} and the chain rule
	$\frac{d}{dx}=\frac{d\theta}{dx}\frac{d}{d\theta}=-\frac{1}{\sin\theta}\frac{d}{d\theta}$  
	yields the trigonometric form
	\begin{equation}\label{eq:crcos}
	P_{k}^{(2)}(\cos\theta)
	=\frac{\sin(k+2)\theta\ \cos\theta-(k+2)\ \cos(k+2)\theta\ \sin\theta}{2\sin^3\theta}.
	\end{equation}
Let $\phi\ceq (r-1)\theta$. From  \eqref{eq:def-Ar-Br}  and \eqref{eq:trigo}, we obtain
\begin{equation}\label{eq:Ar-cos-expand}
	A_r(\cos\theta)
	=\frac{r-2}{2}\cos r\theta-\frac{r}{2}\cos (r-2)\theta.
\end{equation}
Recall the addition formulas
\[
\cos r\theta =\cos(\phi+\theta)=\cos\phi\ \cos\theta-\sin\phi\ \sin\theta,
\]
\[
\cos (r-2)\theta =\cos(\phi-\theta)=\cos\phi\ \cos\theta+\sin\phi\ \sin\theta.
\]
Substitute these into \eqref{eq:Ar-cos-expand}:
\begin{align}
	A_r(\cos\theta)
	&=\frac{r-2}{2}\big(\cos\phi\ \cos\theta-\sin\phi\ \sin\theta\big)
	-\frac{r}{2}\big(\cos\phi\ \cos\theta+\sin\phi\ \sin\theta\big) \nonumber \\[0.3em]
	&=\Big(\frac{r-2}{2}-\frac{r}{2}\Big)\cos\phi\ \cos\theta
	+\Big(-\frac{r-2}{2}-\frac{r}{2}\Big)\sin\phi\ \sin\theta \nonumber \\[0.3em]
	&=-\cos\phi\ \cos\theta-(r-1)\sin\phi\ \sin\theta.\label{eq:arcos}
\end{align}
Define
\[
z\ceq \big(\cos\theta-i(r-1)\ \sin\theta\big)\,e^{i\phi}.
\]
Then
\[
|z|^2
=
\big|\cos\theta-i(r-1)\ \sin\theta\big|^2
=
\cos^2\theta+(r-1)^2\ \sin^2\theta
=
1+r(r-2)\ \sin^2\theta.
\]
Moreover, expanding $z$ reveals its and real and imaginary parts
\[
\Re z=\cos\theta\ \cos\phi+(r-1)\ \sin\theta\ \sin\phi,
\qquad
\Im z=\cos\theta\ \sin\phi-(r-1)\ \sin\theta\ \cos\phi.
\]
By \eqref{eq:def-Ar-Br}-\eqref{eq:crcos}-\eqref{eq:arcos} we get
\[
A_r(\cos\theta)^2 = (\Re z)^2,
\qquad
4\sin^6\theta\,B_r(\cos\theta)^2 = (\Im z)^2.
\]
Therefore
\[
A_r(\cos\theta)^2 + 4(1-\cos^2\theta)^3\ B_r(\cos\theta)^2
= (\Re z)^2+(\Im z)^2 = |z|^2 = 1+r(r-2)\ \sin^2\theta.
\]
Since $\sin^2\theta=1-\cos^2\theta=1-x^2$, we obtain \eqref{eq:identity} at $x=\cos\theta$.
Both sides are polynomials in $x$ and agree on infinitely many $x\in[-1,1]$, hence they agree identically on $\R$.
\end{proof}
	\subsection{Expressions for low relaxation orders}
\begin{figure}[h!]
	\centering
	\includegraphics[width=\textwidth]{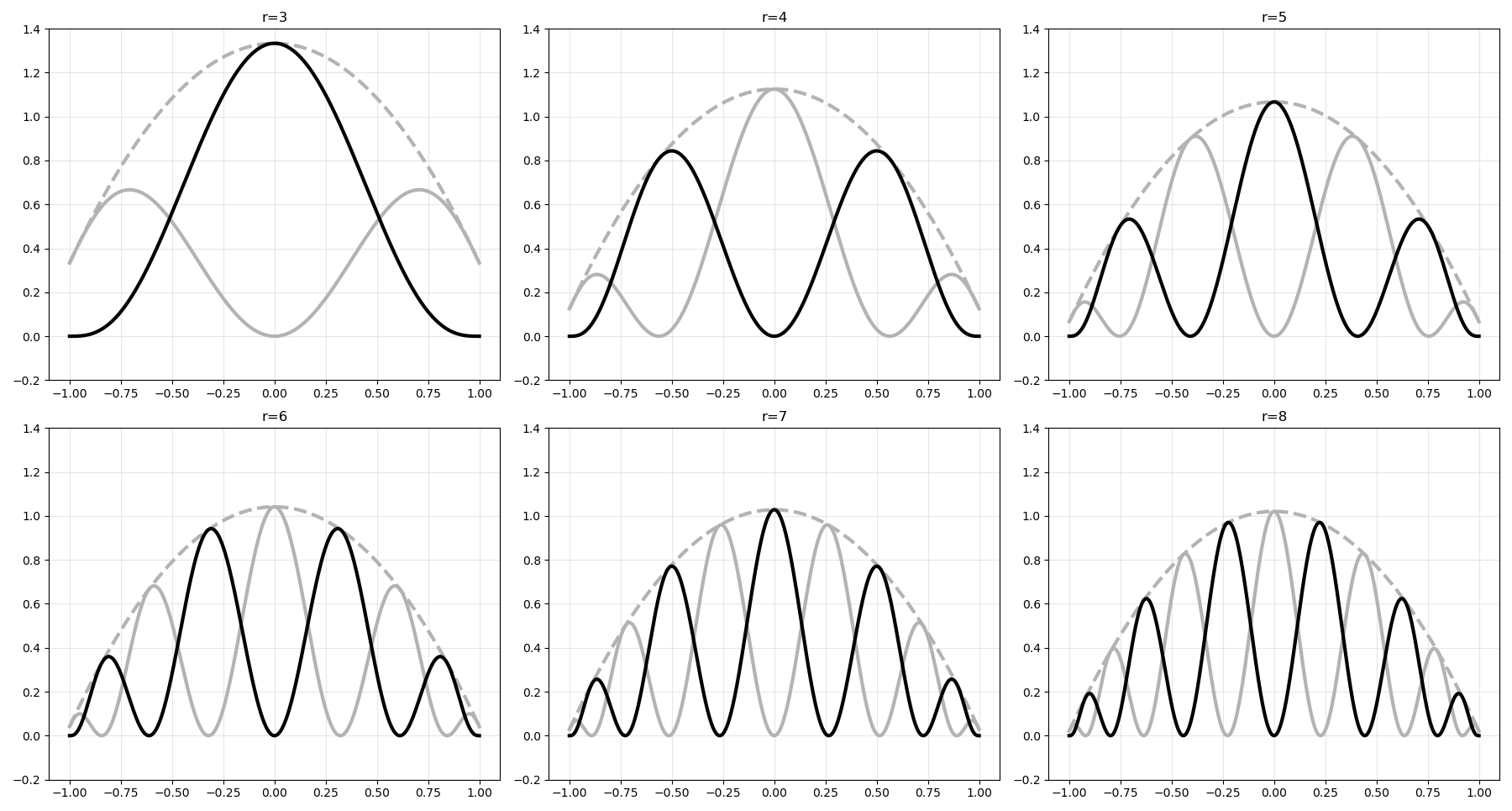}
	\caption{For each relaxation order $r\in\{3,\dots,8\}$ we represent, over $x\in[-1,1]$, the SOS identity \eqref{eq:certificate}.
		The dashed gray curve is the shifted objective $1-x^2+1/r(r-2)$, the plain gray curve is the term $A_r(x)^2/r(r-2)$, and the black curve is the term $4(1-x^2)^3B_r(x)^2/r(r-2)$. The value $-\varepsilon_r = 1/r(r-2)$ can be read from the gray curves at $x=\pm 1$.}
	\label{fig:stengle-sos-identity}
\end{figure}	
	For illustration, the polynomials $(A_r,B_r)$ in the monomial basis are:
	\[
	\begin{array}{rcl}
		A_3(x)&=&2x^3-3x,\\
		A_4(x)&=&8x^4-12x^2+3,\\
		A_5(x)&=&24x^5-40x^3+15x,\\
		A_6(x)&=&64x^6-120x^4+60x^2-5,\\
		A_7(x)&=&160x^7-336x^5+210x^3-35x,\\
		A_8(x)&=&384x^8-896x^6+672x^4-168x^2+7,
	\end{array}
	\qquad
	\begin{array}{rcl}
		B_3(x)&=&1,\\
		B_4(x)&=&4x,\\
		B_5(x)&=&12x^2-2,\\
		B_6(x)&=&32x^3-12x,\\
		B_7(x)&=&80x^4-48x^2+3,\\
		B_8(x)&=&192x^5-160x^3+24x.
	\end{array}
	\]
	The corresponding SOS identities \eqref{eq:certificate} are visualized on Figure \ref{fig:stengle-sos-identity}.

	\section{The moment side: atoms, weights, and positivity}\label{sec:Moments}
	
	In this section we construct a primal-feasible moment sequence $y$ attaining the value
	$\ell_y(1-x^2)=-1/(r(r-2))$.
	
	\subsection{Complementarity suggests atomic support on the roots of $A_r$}
	 
	The identity \eqref{eq:identity} implies that if $(\varepsilon^\star_r,s_0^\star,s_1^\star)$ would be an optimal certificate then
	\[
	\varepsilon_r^\star = -\frac{1}{r(r-2)},
	\qquad
	s_0^\star(x)=\frac{1}{r(r-2)}A_r(x)^2,
	\qquad
	s_1^\star(x)=\frac{4}{r(r-2)}B_r(x)^2.
	\]
	Complementarity with a primal optimal $y$ would read
	\[
	\ell_y(s_0^\star)=0,
	\qquad
	\ell_y(g s_1^\star)=0.
	\]
	Since $s_0^\star$ and $g s_1^\star$ are nonnegative polynomials, these equalities force
	$s_0^\star=0$ and $g s_1^\star=0$ on the support of any representing measure of $y$.
	In particular, $\supp\mu\subseteq\{x \in \R :A_r(x)=0\}$ is a natural candidate.
	
	\subsection{Roots of $A_r$ and their location}

	A key structural fact is that $A_r$ is (up to scaling) a Jacobi polynomial $P^{(\alpha,\beta)}_r$ with parameters
	$\alpha=\beta=-\tfrac32$. See e.g. \cite[Sec. 2.4]{Szego1975} for the definition of Jacobi polynomials, and \cite[Chap. 4]{Szego1975} for their properties.
	
\begin{lemma}[Jacobi representation]\label{lem:Ar-Jacobi}
	For every $r\ge 3$,
	\[
	A_r(x)=\frac{2\sqrt{\pi}\,\Gamma(r+1)}{\Gamma(r-\tfrac12)}\,P_r^{(-3/2,-3/2)}(x)
	\]
	where $\Gamma$ is the Gamma function satisfying  $\Gamma(\tfrac12)=\sqrt{\pi}$, $\Gamma(1)=1$ and $\Gamma(k+1)=k!$ for integer $k$.
\end{lemma}

\begin{proof}
	Fix $\theta\in\R$ and $x=\cos\theta$. First let us establish a derivative identity for $A_r$.
	Using the classical trigonometric representations \eqref{eq:trigo},
	differentiate $T_r(\cos\theta)=\cos(r\theta)$ with respect to $\theta$ and use
	$\frac{dx}{d\theta}=-\sin\theta$ to obtain the standard derivative formula
	\begin{equation}\label{eq:chebder}
	\frac{d}{dx}T_r(x)=r\,U_{r-1}(x).
	\end{equation}
	Using \eqref{eq:def-Ar-Br}, a direct differentiation and a short simplification yield
	\begin{equation}\label{eq:Arprime}
		\frac{d}{dx}A_r(x)=r(r-2)\,T_{r-1}(x).
	\end{equation}
	Chebyshev polynomials have the following Jacobi representation:
	\begin{equation}\label{eq:T-as-Jacobi}
		T_k(x)
		=
		\frac{k!\sqrt{\pi}}{\Gamma(k+\tfrac12)}\,P_k^{(-1/2,-1/2)}(x),
	\end{equation}
	see \cite[Ch.~22, \S22.5, Eq.~(22.5.31)]{AbramowitzStegun1964} and \cite[Chap. 4]{Szego1975}.
	Applying \eqref{eq:T-as-Jacobi} with $k=r-1$ in \eqref{eq:Arprime} gives
	\begin{equation}\label{eq:Arprime-Jacobi}
		\frac{d}{dx}A_r(x)
		=
		\frac{r(r-2)\,\Gamma(r)\sqrt{\pi}}{\Gamma(r-\tfrac12)}\,
		P_{r-1}^{(-1/2,-1/2)}(x).
	\end{equation}
	For Jacobi polynomials one has the derivative relation
	\[
		\frac{d}{dx}P_r^{(\alpha,\beta)}(x)
		=
		\frac12\,(r+\alpha+\beta+1)\,P_{r-1}^{(\alpha+1,\beta+1)}(x),
	\]
	see \cite[Eq.~(4.5.5)]{Szego1975}.
	With $\alpha=\beta=-\tfrac32$, it becomes
	\begin{equation}\label{eq:Jacobi-derivative-special}
		\frac{d}{dx}P_r^{(-3/2,-3/2)}(x)
		=
		\frac12\,(r-2)\,P_{r-1}^{(-1/2,-1/2)}(x).
	\end{equation}
	Comparing \eqref{eq:Arprime-Jacobi} with \eqref{eq:Jacobi-derivative-special}, we see that
	\[
	\frac{d}{dx} A_r(x)
	=
	\Bigg(\frac{2\sqrt{\pi}\,\Gamma(r+1)}{\Gamma(r-\tfrac12)}\Bigg)\;
	\frac{d}{dx}P_r^{(-3/2,-3/2)}(x).
	\]
	Therefore there exists a constant $c_r\in\R$ such that
	\begin{equation}\label{eq:Ar-affine}
		A_r(x)=\frac{2\sqrt{\pi}\,\Gamma(r+1)}{\Gamma(r-\tfrac12)}\,P_r^{(-3/2,-3/2)}(x)+c_r.
	\end{equation}
	Recall the standard normalization
	\[
	P_r^{(-3/2,-3/2)}(1)
	=
	\frac{\Gamma(r-\tfrac12)}{\Gamma(-\tfrac12)\,\Gamma(r+1)}
	\]
	see \cite[Ch.~22, \S22.2, Table~22.2.1]{AbramowitzStegun1964}.
	Evaluating \eqref{eq:Ar-affine} at $x=1$ gives
	\[
	A_r(1)
	=
	\frac{2\sqrt{\pi}\,\Gamma(r+1)}{\Gamma(r-\tfrac12)}\,
	\frac{\Gamma(r-\tfrac12)}{\Gamma(-\tfrac12)\,\Gamma(r+1)}
	+c_r
	=
	\frac{2\sqrt{\pi}}{\Gamma(-\tfrac12)}+c_r.
	\]
	Using $\Gamma(-\tfrac12)=-2\sqrt{\pi}$, we get  $A_r(1)=-1+c_r$. From \eqref{eq:def-Ar-Br}, it holds  $A_r(1)=-1$ and hence $c_r=0$, and \eqref{eq:Ar-affine} reduces to the desired proportionality.
\end{proof}

	\begin{lemma}[Root distribution]\label{lem:roots}
		The roots of $A_r$ are symmetric with respect to the origin, and zero is a root if $r$ is odd. The roots are all real and simple. Moreover, $r-2$ roots lie inside $(-1,1)$ and 2 roots lie outside $[-1,1]$.
	\end{lemma}
	
	\begin{proof}
				Since $T_k(-x)=(-1)^k T_k(x)$, the definition \eqref{eq:def-Ar-Br} gives
		\[
		A_r(-x)=\frac{r-2}{2}(-1)^rT_r(x)-\frac{r}{2}(-1)^{r-2}T_{r-2}(x)
		=(-1)^rA_r(x)
		\]
		which shows that the roots are symmetric with respect to the origin, and also that $A_r(0)=0$ if $r$ is odd.
By Lemma~\ref{lem:Ar-Jacobi}, $A_r$ and $P_r^{(-3/2,-3/2)}$ have the same roots. Then we apply \cite[Thm.~6.72]{Szego1975} with the notations $\alpha=\beta=-\tfrac32$ and $n:=r$.
First, the excluded cases (6.72.1)--(6.72.3) do not occur: $\alpha,\beta$ are not negative
integers, and $n+\alpha+\beta=n-3\ge 0$ for $n\ge 3$. Hence the roots are different from
$\pm1$ and $\infty$, and (by the discussion following (6.72.3)) are distinct.
Next, compute the integers $X,Y,Z$ from (6.72.5) using Klein's symbol $E(\cdot)$ in (6.72.4).
Here $2n+\alpha+\beta+1=2n-2>0$ and $|\alpha|=|\beta|=\tfrac32$, so
\[
X=E\!\left(\frac12\bigl((2n-2)-\tfrac32-\tfrac32+1\bigr)\right)=E(n-2)=n-3,
\qquad
Y=Z=E\!\left(\frac12(-2n+3)\right)=0.
\]
Moreover,
\[
\binom{n+\alpha}{n}=\binom{n-\tfrac32}{n}
=\frac{\Gamma(n-\tfrac12)}{\Gamma(n+1)\Gamma(-\tfrac12)}<0,
\qquad
\binom{2n+\alpha+\beta}{n}=\binom{2n-3}{n}>0.
\]
Therefore, in (6.72.6) we obtain $N_1=n-2$ (both parity cases give the same result),
while in (6.72.7)--(6.72.8) we are in the strictly negative branches and get $N_2=N_3=1$.
Thus $P_n^{(-3/2,-3/2)}$ has exactly $n-2$ roots in $(-1,1)$, one in $(-\infty,-1)$
and one in $(1,\infty)$, and they are all simple.
	\end{proof}
	
	Based on Lemma \ref{lem:roots}, denoting by $\{x_i\}_{i=1}^r$  the roots of $A_r$, their pattern is as follows:
		\[
		\begin{array}{ll}
			r\ \text{even:}& 0<x_1<x_2<\cdots<x_{\frac{r}{2}-1}<1<x_{\frac{r}{2}}=:x_{\mathrm{out}},\\[0.25em]
			r\ \text{odd:}& 0=x_1<x_2<\cdots<x_{\frac{r-1}{2}}<1<x_{\frac{r+1}{2}}=:x_{\mathrm{out}}.
		\end{array}
		\]
	Graphs of polynomials $A_r$ with their root distribution are displayed on Figure \ref{fig:stengle_factors}.
\begin{figure}[h!]
	\centering
	\includegraphics[width=\textwidth]{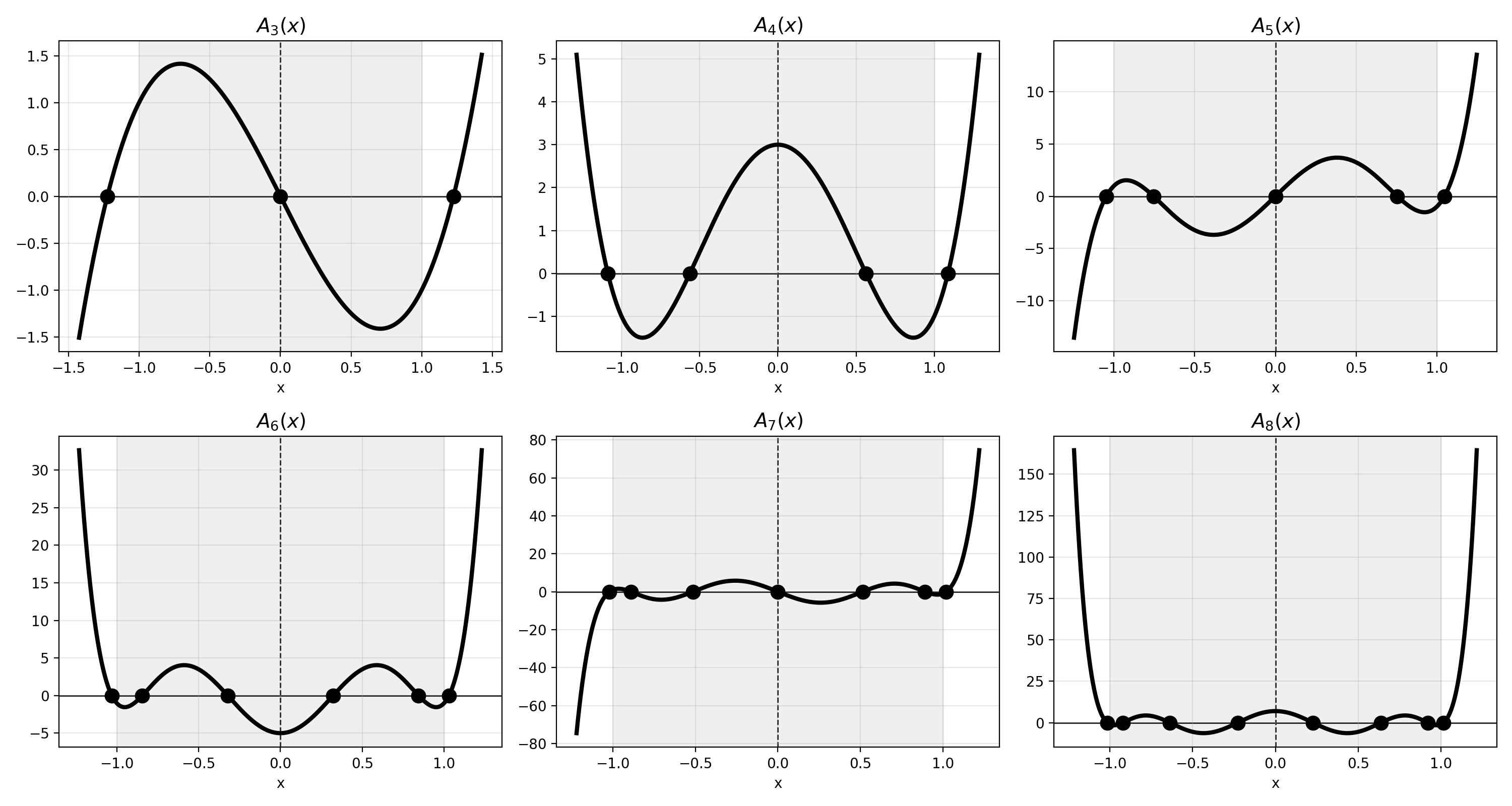}
	\caption{Graphs of polynomials $A_r$ (thick black lines) and root distribution (thick black dots) for each relaxation order $r\in\{3,\dots,8\}$. We observe that 2 of the $r$ roots are outside the interval $[-1,1]$ (light gray region).}
	\label{fig:stengle_factors}
\end{figure}	
	\subsection{The atomic measure and its basic moment identities}
	
	With $\{x_i\}_{i=1}^r$ denoting the roots of $A_r$, define the finitely atomic measure
	\begin{equation}\label{eq:mu}
		\mu(dx)\ceq \frac{1}{(r(r-2))^2}\sum_{i=1}^r \frac{1}{(1-x_i^2)^2}\,\delta_{x_i}(dx).
	\end{equation}
	Let $y=\{y_k\}_{k=0}^{2r}$ be its moment sequence: $y_k=\int x^k\,\mu(dx)$.
	
	\begin{lemma}[Two residue identities]\label{lem:residue-sums}
		It holds
		\begin{equation}\label{eq:sum-1over1mx2sq}
			\sum_{i=1}^r \frac{1}{(1-x_i^2)^2} = (r(r-2))^2,
		\end{equation}
		and
		\begin{equation}\label{eq:sum-1over1mx2}
			\sum_{i=1}^r \frac{1}{1-x_i^2} = -r(r-2).
		\end{equation}
	\end{lemma}
	
\begin{proof}
	The proof uses the classical residue formulas of complex analysis \cite[Chap. 3, Sec. 2]{Stein2003}. 
	Let $F$ be a holomorphic function on a punctured neighborhood $\{z \in \C : 0<|z-a|<r\}$ and admit a Laurent expansion
	\[
	F(z)=\sum_{k=-m}^{\infty} f_k (z-a)^k .
	\]
	The \emph{residue} of $F$ at $a$ is the coefficient of $(z-a)^{-1}$ in this expansion, i.e.
	\[
	\Res_{z=a} F(z) := f_{-1}.
	\]
	Equivalently, if $a$ is a pole of order $n$ of $F$, then by \cite[Ch.~3, Thm.~1.4]{Stein2003},
	\begin{equation}\label{eq:resn}
	\Res_{z=a} F(z)
	=\lim_{z\to a}\frac{1}{(n-1)!}\,\frac{d^{\,n-1}}{dz^{\,n-1}}\Big((z-a)^n F(z)\Big).
	\end{equation}
	In particular, for a simple pole ($n=1$),
	\begin{equation}\label{eq:res}
	\Res_{z=a} F(z)=\lim_{z\to a}(z-a)F(z).
	\end{equation}
	
	\medskip
	\noindent{1) Proof of \eqref{eq:sum-1over1mx2sq}.}
	Let
	\[
	F(z):=\frac{A_r'(z)}{A_r(z)}\frac{1}{(1-z^2)^2},
	\qquad
	G(z):=\frac{A_r'(z)}{A_r(z)}.
	\]
	Let $\mathcal P \subset \R$ denote the poles of $F$, which are the (simple) zeros $x_i$ of $A_r$ and the points $\pm 1$,
	which are poles of order $2$.

	Let $\Gamma_R=\{z \in \C:|z|=R\}$ with $R>0$ large enough so that all poles lie inside.
	By the residue formula \cite[Ch.~3, \S2]{Stein2003},
	\[
	\int_{\Gamma_R}F(z)\,dz = 2\pi i \sum_{a \in \mathcal P}\Res_{z=a}F(z).
	\]
	As $z\to\infty$, $G(z)=O(1/z)$ and $(1-z^2)^{-2}=O(1/z^4)$, hence $F(z)=O(1/z^5)$.
	Thus $\int_{\Gamma_R}F(z)\,dz\to 0$ as $R\to\infty$, and we obtain
	\begin{equation}\label{eq:sumresFzero-SS-short}
		\sum_{a \in \mathcal P}\Res_{z=a}F(z)=0.
	\end{equation}
	
	Each $x_i$ is a simple zero of $A_r$, so $A_r'(x_i)\neq 0$ and $F$ has a simple pole at $x_i$.
	Using \eqref{eq:res}:
	\[
	\Res_{z=x_i}F(z)=\lim_{z\to x_i}(z-x_i)F(z)=\frac{1}{(1-x_i^2)^2}.
	\]
	At $z=1$, $F$ has a pole of order $2$, hence from \eqref{eq:resn}:
	\[
	\Res_{z=1}F(z)=\lim_{z\to 1}\frac{d}{dz}\Big((z-1)^2F(z)\Big) =
	\left.\frac{d}{dz}\Big(\frac{G(z)}{(1+z)^2}\Big)\right|_{z=1}
	=\frac{G'(1)-G(1)}{4}.
	\]
	Similarly,
	\[
	\Res_{z=-1}F(z)=\lim_{z\to -1}\frac{d}{dz}\Big((z+1)^2F(z)\Big)
	=\left.\frac{d}{dz}\Big(\frac{G(z)}{(1-z)^2}\Big)\right|_{z=-1}
	=\frac{G'(-1)+G(-1)}{4}.
	\]	
	Now $A_r(1)=-1$, $A_r(-1)=-(-1)^r$ and from \eqref{eq:Arprime} we have $A_r'(1)=r(r-2)$, $A_r'(-1)=r(r-2)(-1)^{r-1}$ which implies
	\[
	G(1)=-r(r-2), \qquad G(-1)=r(r-2).
	\]
	Moreover $G'(z)=A_r''(z)/A_r(z)-(A_r'(z)/A_r(z))^2$.
	Using \eqref{eq:chebder} and \eqref{eq:Arprime} we have $A_r'(x)=r(r-2)T_{r-1}(x)$, $T_{r-1}'(x)=(r-1)U_{r-2}(x)$, and
	$U_{r-2}(1)=r-1$, $U_{r-2}(-1)=(-1)^{r-2}(r-1)$. So we get
	$A_r''(1)=r(r-2)(r-1)^2$,
	$A_r''(-1)=r(r-2)(-1)^r(r-1)^2$,
	hence
	\[
	G'(1)= -r(r-2)(r-1)^2-r^2(r-2)^2,\qquad
	G'(-1)= -r(r-2)(r-1)^2-r^2(r-2)^2.
	\]
	Therefore $G'(1)-G(1)=-2r^2(r-2)^2$ and $G'(-1)+G(-1)=-2r^2(r-2)^2$, so
	\[
	\Res_{z=1}F(z)=\Res_{z=-1}F(z)=-\frac{r^2(r-2)^2}{2}.
	\]
	
	Substituting the residues into \eqref{eq:sumresFzero-SS-short} yields
	\[
		\sum_{a \in \mathcal P}\Res_{z=a}F(z) = \sum_{i=1}^r\frac{1}{(1-x_i^2)^2}
	-\frac{r^2(r-2)^2}{2}-\frac{r^2(r-2)^2}{2}=0,
	\]
	which is \eqref{eq:sum-1over1mx2sq}.
	
	\medskip
	\noindent{2) Proof of \eqref{eq:sum-1over1mx2}.}
	Define
	\[
	H(z):=\frac{A_r'(z)}{A_r(z)}\frac{1}{1-z^2}=G(z)\frac{1}{(1-z)(1+z)}.
	\]
	The poles are simple at $x_i$ and at $\pm1$.
	As $z\to\infty$, $H(z)=O(1/z^3)$, so the same contour argument gives
	\[
	\sum_{a\in\mathcal P}\Res_{z=a}H(z)=0.
	\]
	By \eqref{eq:res}:
	\[
	\Res_{z=x_i}H(z)=\lim_{z\to x_i}(z-x_i)H(z)=\frac{1}{1-x_i^2}.
	\]
	Also,
	\[
	\Res_{z=1}H(z)=\lim_{z\to 1}(z-1)H(z)=\lim_{z\to 1}G(z)\frac{z-1}{(1-z)(1+z)}
	=-\frac{G(1)}{2}=\frac{r(r-2)}{2},
	\]
	and
	\[
	\Res_{z=-1}H(z)=\lim_{z\to -1}(z+1)H(z)=\lim_{z\to -1}G(z)\frac{z+1}{(1-z)(1+z)}
	=\frac{G(-1)}{2}=\frac{r(r-2)}{2}.
	\]
	Thus
	\[
		\sum_{a \in \mathcal P}\Res_{z=a}H(z) =  \sum_{i=1}^r\frac{1}{1-x_i^2}+\frac{r(r-2)}{2}+\frac{r(r-2)}{2}=0,
	\]
	which is \eqref{eq:sum-1over1mx2}.
\end{proof}

	\begin{lemma}[Probability and objective value]\label{lem:prob-and-obj}
		The measure $\mu$ in \eqref{eq:mu} satisfies
		\[
		y_0 = \int d\mu(x) = 1, \qquad
		\ell_y(1-x^2) = \int (1-x^2)\,d\mu(x) = -\frac{1}{r(r-2)}.
		\]
	\end{lemma}
	
	\begin{proof}
		The mass is
		\[
		\int d\mu(x) = \frac{1}{(r(r-2))^2}\sum_{i=1}^r \frac{1}{(1-x_i^2)^2}=1
		\]
		by \eqref{eq:sum-1over1mx2sq}.
		Next,
		\[
		\int (1-x^2)\,d\mu(x)
		=\frac{1}{(r(r-2))^2}\sum_{i=1}^r \frac{1-x_i^2}{(1-x_i^2)^2}
		=\frac{1}{(r(r-2))^2}\sum_{i=1}^r \frac{1}{1-x_i^2}
		= -\frac{1}{r(r-2)}
		\]
		by \eqref{eq:sum-1over1mx2}.
	\end{proof}

\subsection{Positivity of the moment matrix}

\begin{lemma}\label{lem:moment-matrix-psd}
	Let $\mu$ be the measure \eqref{eq:mu} with moment vector $y$. Then its moment matrix of order $r$  is positive semidefinite: 
	\[
	M_r(y)\succeq 0.
	\]
\end{lemma}

\begin{proof}
Given any polynomial $p \in \R[x]_r$, 
using the definition  \eqref{eq:mommat} of the moment matrix:
\[
p^\top M_r(y)\,p =  \ell_y(p^2) = 
\int p(x)^2\,d\mu(x).
\]
Since $\mu$ is a nonnegative measure, it holds
\[
\int p(x)^2\,d\mu(x)\ge 0,
\]
which implies $M_r(y)\succeq 0$.
\end{proof}

\subsection{Positivity of the localizing matrix}

\begin{lemma}\label{lem:psd-localizing}
	Let $\mu$ be the measure \eqref{eq:mu} with moment vector $y$. Then its localizing moment matrix of order $r-3$ is positive semidefinite: 
	\[
	M_{r-3}(g y)\succeq 0.
	\]
\end{lemma}

\begin{proof}
	Set $n:=r-3$ and write $b(x):=(1,x,\dots,x^n)^\top$.
	By \eqref{eq:locmat}, $M_n(g\,y)\succeq0$ is equivalent to
	\[
	\int (1-x^2)^3 p(x)^2\,d\mu(x)\ge 0\qquad\forall\,p\in\mathbb R[x]_n.
	\]
	Using \eqref{eq:mu} it holds
	\[
	\int (1-x^2)^3 p(x)^2\,d\mu(x)
	=\sum_{i=1}^r \frac{1-x_i^2}{(r(r-2))^2}\,p(x_i)^2.
	\]
	Form Lemma \ref{lem:roots} we know that among the $r$ roots of $A_r$, exactly $r-2$ satisfy $|x_i|<1$ and two are $\pm x_{\rm out}$ with
	$x_{\rm out}>1$. Define the positive weights
	\[
	w_i:=\frac{1-x_i^2}{(r(r-2))^2}>0\quad(|x_i|<1),\qquad
	w_{\rm out}:=\frac{x_{\rm out}^2-1}{(r(r-2))^2}>0,
	\]
	and the matrices
	\[
	H:=\sum_{|x_i|<1} w_i\,b(x_i)b(x_i)^\top,\qquad
	V:=\sqrt{w_{\rm out}}\,[\,b(x_{\rm out})\ \ b(-x_{\rm out})\,].
	\]
	Then the localizing matrix decomposes as
	\begin{equation}\label{eq:MHV}
		M_n(g\,y)=H-VV^\top.
	\end{equation}
	Since the inner measure $\sum_{|x_i|<1}w_i\delta_{x_i}$ has $n+1$ distinct nodes in $(-1,1)$ with
	positive weights, its moment matrix is positive definite, i.e. $H\succ0$. Therefore, by a rank-two Schur complement, we have the equivalence
	\begin{equation}\label{eq:SchurG}
	M_n(gy) =
H-VV^\top\succeq 0
\quad\Longleftrightarrow\quad
\left(\begin{array}{cc}H & V \\ V^\top & I_2 \end{array}\right) \succeq 0 	
\quad\Longleftrightarrow\quad
G\ceq V^\top H^{-1}V \preceq I_2.
	\end{equation}
	
	Define the Christoffel-Darboux kernel of the inner measure
	\[
	K_n(x,y):=b(x)^\top H^{-1}b(y),
	\]
	see e.g.\ \cite[Chap.~2]{Lasserre2022}. Then
	\[
	G
	=w_{\rm out}
	\begin{pmatrix}
		K_n(x_{\rm out},x_{\rm out}) & K_n(x_{\rm out},-x_{\rm out})\\
		K_n(x_{\rm out},-x_{\rm out}) & K_n(x_{\rm out},x_{\rm out})
	\end{pmatrix}
	=
	\begin{pmatrix}\alpha&\beta\\ \beta&\alpha\end{pmatrix},
	\]
	with $\alpha:=w_{\rm out}K_n(x_{\rm out},x_{\rm out})$ and $\beta:=w_{\rm out}K_n(x_{\rm out},-x_{\rm out})$.
	Hence
	\[
	\lambda_+=\alpha+\beta\ \text{with eigenvector}\ u_+=(1,1),\qquad
	\lambda_-=\alpha-\beta\ \text{with eigenvector}\ u_-=(1,-1).
	\]
	
	Let $q(x):=B_r(x)=P_n^{(2)}(x)$ and let $q$ be its coefficient vector in the basis $b$, i.e.
	$q(x)=q^\top b(x)$. We claim that
	\begin{equation}\label{eq:Mq0-short}
		M_n(g\,y)\,q=0.
	\end{equation}
	Indeed, it suffices to show $\ell_y(g\,x^k q)=0$ for $k=0,\dots,n$. Using \eqref{eq:mu},
	\[
	\ell_y(g\,x^k q)=\frac{1}{(r(r-2))^2}\sum_{i=1}^r (1-x_i^2)\,x_i^k\,q(x_i),
	\]
	so we have to prove the discrete orthogonality relation
	\begin{equation}\label{eq:disc-orth-short}
		\sum_{i=1}^r (1-x_i^2)\,x_i^k\,B_r(x_i)=0,\qquad k=0,\dots,n.
	\end{equation}

Recall the two polynomial identities derived from the trigonometric formulas \eqref{eq:trigo}, \eqref{eq:crcos} and \eqref{eq:arcos}:
\begin{align}
	A_r(x)
	&= -\Big(x\,T_{r-1}(x) + (r-1)(1-x^2)\,U_{r-2}(x)\Big), \label{eq:Ar-TU-recall}\\
	2(1-x^2)\,B_r(x)
	&= x\,U_{r-2}(x) - (r-1)\,T_{r-1}(x). \label{eq:Br-TU-recall}
\end{align}
Let $x_i$ be such that $A_r(x_i)=0$. Since $x_i\neq \pm1$, \eqref{eq:Ar-TU-recall} gives
\[
x_i\,T_{r-1}(x_i)=-(r-1)(1-x_i^2)\,U_{r-2}(x_i),
\qquad\text{hence}\qquad
U_{r-2}(x_i)=-\frac{x_i}{(r-1)(1-x_i^2)}\,T_{r-1}(x_i).
\]
Substitute this into \eqref{eq:Br-TU-recall} at $x=x_i$:
\begin{align*}
	2(1-x_i^2)\,B_r(x_i)
	&=x_i\,U_{r-2}(x_i)-(r-1)T_{r-1}(x_i)
	=-\left(\frac{x_i^2}{(r-1)(1-x_i^2)}+(r-1)\right)T_{r-1}(x_i)\\
	&=-\frac{x_i^2+(r-1)^2(1-x_i^2)}{(r-1)(1-x_i^2)}\,T_{r-1}(x_i)
	=-\frac{1+r(r-2)(1-x_i^2)}{(r-1)(1-x_i^2)}\,T_{r-1}(x_i),
\end{align*}
where we used the algebraic identity
\(
x^2+(r-1)^2(1-x^2)=1+r(r-2)(1-x^2).
\)
Multiplying by $(r-1)(1-x_i^2)$ yields
\begin{equation}\label{eq:pre-pointwise}
	2(r-1)(1-x_i^2)^2\,B_r(x_i)=-\bigl(1+r(r-2)(1-x_i^2)\bigr)\,T_{r-1}(x_i).
\end{equation}

Now evaluate the SOS identity \eqref{eq:identity} at $x=x_i$. Since $A_r(x_i)=0$, it reduces to
\begin{equation}\label{eq:identity-at-root}
	1+r(r-2)(1-x_i^2)=4(1-x_i^2)^3\,B_r(x_i)^2.
\end{equation}
Insert \eqref{eq:identity-at-root} into \eqref{eq:pre-pointwise}:
\[
2(r-1)(1-x_i^2)^2\,B_r(x_i)
=-4(1-x_i^2)^3\,B_r(x_i)^2\,T_{r-1}(x_i).
\]
We now justify division by $(1-x_i^2)^2B_r(x_i)$.
First, $x_i\neq\pm1$ so $(1-x_i^2)^2\neq0$.
Second, the Gegenbauer polynomial $B_r=P^{(2)}_{r-3}$ has all its real zeros in $(-1,1)$,
and if $|x_i|<1$ then \eqref{eq:identity-at-root} gives
$1+r(r-2)(1-x_i^2)>1$, so $B_r(x_i)\neq0$ as well.
Hence $B_r(x_i)\neq0$ for every root $x_i$ of $A_r$.
Dividing by $(1-x_i^2)^2B_r(x_i)$ gives
\[
(1-x_i^2)\,T_{r-1}(x_i)\,B_r(x_i)=-\frac{r-1}{2}.
\]
Using the derivative identity \eqref{eq:Arprime}:
\[
A_r'(x)=r(r-2)\,T_{r-1}(x),
\]
one gets the pointwise relation
	\begin{equation}\label{eq:BoverAprime-short}
		(1-x_i^2)\,B_r(x_i)=-\frac{r(r-1)(r-2)}{2}\,\frac{1}{A_r'(x_i)}.
	\end{equation}
	Next, for each $k\le n=r-3$, consider the meromorphic function $F_k(z):=z^k/A_r(z)$. Its only finite
	poles are the simple zeros $x_i$ of $A_r$, and as in the proof of Lemma \ref{lem:residue-sums}, by the simple-pole residue formula
	\cite[Chap.~3, Thm.~1.4]{Stein2003},
	\[
	\Res_{z=x_i}F_k(z)=\frac{x_i^k}{A_r'(x_i)}.
	\]
	Since $\deg A_r=r$ and $k\le r-3$, we have $F_k(z)=O(1/z^3)$ as $z\to\infty$; hence the contour integral
	over $|z|=R$ vanishes as $R\to\infty$. By the residue theorem \cite[Chap.~3, Sec.~2]{Stein2003},
	\[
	0=\sum_{i=1}^r \Res_{z=x_i}F_k(z)=\sum_{i=1}^r \frac{x_i^k}{A_r'(x_i)}.
	\]
	Multiplying by $-\frac{r(r-1)(r-2)}{2}$ and using \eqref{eq:BoverAprime-short} yields \eqref{eq:disc-orth-short},
	hence \eqref{eq:Mq0-short} holds.
	
	Now combine \eqref{eq:MHV} and \eqref{eq:Mq0-short}:
	\[
	(H-VV^\top) q=0\quad\Longrightarrow\quad H  q=VV^\top q.
	\]
	Left-multiplying by $V^\top H^{-1}$ gives
	\[
	V^\top q=(V^\top H^{-1}V)(V^\top  q)=G\,(V^\top  q),
	\]
	so $1$ is an eigenvalue of $G$ with eigenvector $V^\top q\neq 0$.
	Moreover, since $q(-x)=(-1)^n q(x)$,
	\[
	V^\top q
	=\sqrt{w_{\rm out}}\,(q(x_{\rm out}),q(-x_{\rm out}))^\top
	=\sqrt{w_{\rm out}}\,q(x_{\rm out})\,(1,(-1)^n)^\top,
	\]
	so the eigenvector is precisely $u_{(-1)^n}=(1,(-1)^n)$ and therefore
	\begin{equation}\label{eq:eig1}
		\lambda_{(-1)^n}=1.
	\end{equation}
	
	It remains to show $\max(\lambda_+,\lambda_-)=1$. For this we use only the sign of $\beta$.
	Let $\{\pi_k\}$ be the monic orthogonal polynomials for the inner product
	$\langle f,g\rangle_{\rm inn}:=\sum_{|x_i|<1}w_i f(x_i)g(x_i)$ and $h_k:=\langle\pi_k,\pi_k\rangle_{\rm inn}>0$.
	The Christoffel-Darboux formula \cite[\S 3.1.1]{Lasserre2022} gives, for $x\neq y$,
	\[
	K_n(x,y)=\frac{1}{h_n}\,\frac{\pi_{n+1}(x)\pi_n(y)-\pi_n(x)\pi_{n+1}(y)}{x-y}.
	\]
	Since $\nu_{\rm inn}$ is symmetric, $\pi_k(-x)=(-1)^k\pi_k(x)$, so with $(x,y)=(x_{\rm out},-x_{\rm out})$,
	\[
	K_n(x_{\rm out},-x_{\rm out})
	=\frac{(-1)^n}{x_{\rm out}h_n}\,\pi_n(x_{\rm out})\pi_{n+1}(x_{\rm out}).
	\]
	All zeros of $\pi_k$ lie in $(-1,1)$, hence $\pi_k(x_{\rm out})>0$ for $x_{\rm out}>1$, so
	$K_n(x_{\rm out},-x_{\rm out})\neq 0$ and $\mathrm{sign}(K_n(x_{\rm out},-x_{\rm out}))=(-1)^n$.
	As $w_{\rm out}>0$, we obtain
	\[
	\mathrm{sign}(\beta)=(-1)^n,
	\]
	so the eigenvalue associated with $u_{(-1)^n}$ is exactly $\alpha+|\beta|=\max(\lambda_+,\lambda_-)=1$.
	Together with \eqref{eq:eig1}, this yields
	$\max(\lambda_+,\lambda_-)=1$ hence $G\preceq I_2$.
	Finally, \eqref{eq:SchurG} implies $M_n(g\,y)= M_{r-3}(g\,y)\succeq 0$.
\end{proof}

\subsection{Further properties of the moments}

	\begin{lemma}\label{lem:moments-basic}
		For the moment sequence $y$ of the measure $\mu$ in \eqref{eq:mu}, it holds $y_k=0$ for odd $k$ and $1 + k(y_2-1) \leq  y_{2k}\le y_{2(k+1)}$ for all $k\ge 0$.
	\end{lemma}
	
	\begin{proof}
		The roots of $A_r$ are symmetric: if $x_i$ is a root then so is $-x_i$.
		The weights in \eqref{eq:mu} depend only on $x_i^2$, hence the atom at $x_i$ and the atom at $-x_i$
		have the same weight. Therefore, for any odd integer $k$,
		\[
		y_k=\int x^k\,d\mu(x)
		=\int (-x)^k\,d\mu(x)
		=-\int x^k\,d\mu(x),
		\]
		hence $y_k=0$.
		
		Fix an integer $k\ge 1$ and set $u(x):=x^2\ge 0$. Then $y_{2k}=\int u(x)^k\,d\mu(x)$.
		By convexity of $u\mapsto u^k$ on $[0,\infty)$, its tangent inequality at $u=1$ reads
		\[
		u^k(x) \ge 1 + k(u(x)-1)\qquad \forall x.
		\]
		Integrating this inequality with respect to $\mu$ gives
		\[
		y_{2k}=\int u(x)^k\,d\mu(x)
		\ge
		\int \big(1+k(u(x)-1)\big)\,d\mu(x)
		=
		1+k\Big(\int u(x)\,d\mu(x)-\int 1\,d\mu(x)\Big)
		=
		1+k(y_2-1),
		\]
		where we used Lemma \ref{lem:prob-and-obj} and $\int d\mu=1$. The same Lemma \ref{lem:prob-and-obj} yields
		\[
		y_2=\int x^2\,d\mu(x)=1-\int(1-x^2)\,d\mu(x)=1+\frac{1}{r(r-2)}>1,
		\]
		hence $y_{2k}> 1$ for all $k\ge 1$.
		
		To prove that the even moments are nondecreasing, define two nondecreasing functions on $[0,\infty)$:
		\[
		f(u):=u^{k-1},\qquad g(u):=u-1.
		\]
		Consider the product measure $\mu\otimes\mu$ on $\R\times\R$. Since $f$ and $g$ are nondecreasing,
		for all $x,x'\in\R$ we have the pointwise inequality
		\[
		\big(f(u(x))-f(u(x'))\big)\big(g(u(x))-g(u(x'))\big)\ \ge\ 0.
		\]
		Integrating with respect to $\mu(x)\mu(x')$ gives
		\begin{align*}
			0
			&\le
			\iint \big(f(u(x))-f(u(x'))\big)\big(g(u(x))-g(u(x'))\big)\,d\mu(x)\,d\mu(x')\\
			&=
			2\int f(u(x))g(u(x))\,d\mu(x)
			-2\Big(\int f(u(x))\,d\mu(x)\Big)\Big(\int g(u(x))\,d\mu(x)\Big).
		\end{align*}
		Hence
		\begin{equation}\label{eq:chebyshev-mu}
			\int f(u(x))g(u(x))\,d\mu(x)
			\ \ge\
			\Big(\int f(u(x))\,d\mu(x)\Big)\Big(\int g(u(x))\,d\mu(x)\Big).
		\end{equation}
		Substituting $f(u)=u^{k-1}$ and $g(u)=u-1$ in \eqref{eq:chebyshev-mu} yields
		\[
		\int u(x)^{k-1}(u(x)-1)\,d\mu(x)
		\ \ge\
		\Big(\int u(x)^{k-1}\,d\mu(x)\Big)\Big(\int (u(x)-1)\,d\mu(x)\Big).
		\]
		Now
		\[
		\int u(x)^{k-1}(u(x)-1)\,d\mu(x)
		=\int u(x)^k\,d\mu(x)-\int u(x)^{k-1}\,d\mu(x)
		=y_{2k}-y_{2(k-1)},
		\]
		and
		\[
		\int (u(x)-1)\,d\mu(x)=\int x^2\,d\mu(x)-\int 1\,d\mu(x)=y_2-1.
		\]
		Therefore we obtain the moment gap bound:
		\[
		y_{2k}-y_{2(k-1)} \ \ge\ y_{2(k-1)}(y_2-1).
		\]
		Since $y_{2(k-1)}\ge 0$ and $y_2-1=\frac{1}{r(r-2)}>0$, it follows that
		$y_{2k}\ge y_{2(k-1)}$ for all $k\ge 1$.
	\end{proof}

	\begin{lemma}\label{prop:y2r}
			For the moment sequence $y$ of the measure $\mu$ in \eqref{eq:mu}, it holds  $\lim_{r\to\infty} y_{2r} = 1$.
	\end{lemma}
	
	\begin{proof}
		By Lemma~\ref{lem:moments-basic}, $y_{2r}\ge 1$.
		For the upper bound, note that $\mu$ is supported on $\abs{x}\le x_{\mathrm{out}}$, hence
		\[
		y_{2r}=\int x^{2r}\,d\mu \le x_{\mathrm{out}}^{2r}.
		\]
		It remains to show $x_{\mathrm{out}}^{2r}\to 1$. Writing $x_{\mathrm{out}}=\cosh t_r$ for some $t_r>0$ and using
		$T_k(\cosh t)=\cosh(kt)$, the equation $A_r(x_{\mathrm{out}})=0$ becomes
		\[
		(r-2)\cosh(rt_r) = r\cosh((r-2)t_r).
		\]
		A Taylor expansion of $\cosh$ for small $t_r$ implies $t_r=\Theta(1/r)$ and therefore
		$x_{\mathrm{out}}-1=\cosh t_r-1=\Theta(t_r^2)=\Theta(1/r^2)$.
		Consequently,
		\[
		\log(x_{\mathrm{out}}^{2r}) = 2r\log(1+O(1/r^2)) = O(1/r)\to 0,
		\]
		so $x_{\mathrm{out}}^{2r}\to 1$ and hence $y_{2r}\to 1$.
	\end{proof}

	\subsection{Rational expressions of the moments}
	
	The optimal moments are actual \emph{rational} numbers. To see why,
	define a polynomial $P_r(u)\in\mathbb Q[u]$ by extracting the even part of $A_r$:
	\[
	A_r(x)=
	\begin{cases}
		P_r(x^2),& r\ \text{even},\\
		x\,P_r(x^2),& r\ \text{odd}.
	\end{cases}
	\]
	Then the nonzero squared nodes $u_i>0$ are precisely the roots of $P_r(u)$.
	The crucial point is that the moment sum
	\[
		y_{2p}:=\int x^{2p}\,d\mu_r(x)=w_0 0^p + 2\sum_{u_i> 0}w_i {u_i^p},\qquad p\ge 0
	\]
	is a \emph{symmetric rational function}
	of the roots of an explicit polynomial with rational coefficients.
	
	Let $C_r\in\mathbb Q^{m\times m}$ denote the companion matrix of the monic polynomial proportional to $P_r$ (here $m=\deg P_r=\lfloor r/2\rfloor$), so that the eigenvalues of $C_r$ are
	exactly the roots $(u_i)$ of $P_r$.
	Since $u=1$ is not a root of $P_r$, the matrix $I-C_r$ is invertible, and the spectral mapping theorem yields
	\[
	\mathrm{tr}\!\left(C_r^p\,(I-C_r)^{-2}\right)
	=\sum_{u_i>0}\frac{u_i^p}{(1-u_i)^2}.
	\]
	Therefore, for every $p\ge 1$,
	\begin{equation}
		\label{eq:trace-moment}
		y_{2p}=
		\frac{2}{r^2(r-2)^2}\,\mathrm{tr}\!\left(C_r^p\,(I-C_r)^{-2}\right).
	\end{equation}

	\subsection{Moment vectors for low relaxation orders}\label{subsec:explicit-moments}
	
	Using the trace formula (\ref{eq:trace-moment}) we can compute the first moment vectors.
	
	\paragraph{Order $r=3$.}
	\[
	(y_0,y_2,y_4,y_6)=\left(1,\ \frac{4}{3},\ 2,\ 3\right).
	\]
	
	\paragraph{Order $r=4$.}
	\[
	(y_0,y_2,y_4,y_6,y_8)
	=\left(1,\ \frac{9}{8},\ \frac{21}{16},\ \frac{99}{64},\ \frac{117}{64}\right).
	\]
	
	\paragraph{Order $r=5$.}
	\[
	(y_0,y_2,y_4,y_6,y_8,y_{10})
	=\left(1,\ \frac{16}{15},\ \frac{52}{45},\ \frac{34}{27},\ \frac{223}{162},\ \frac{1465}{972}\right).
	\]
	
	\paragraph{Order $r=6$.}
	\[
	(y_0,y_2,y_4,y_6,y_8,y_{10},y_{12})
	=\left(
	1,\ \frac{25}{24},\ \frac{35}{32},\ \frac{295}{256},\ \frac{7475}{6144},\ \frac{21075}{16384},\ \frac{178425}{131072}
	\right).
	\]
	
	\paragraph{Order $r=7$.}
	\[
	(y_0,y_2,y_4,y_6,y_8,y_{10},y_{12},y_{14})
	=\left(
	1,\ \frac{36}{35},\ \frac{186}{175},\ \frac{963}{875},\ \frac{2853}{2500},\ \frac{29613}{25000},\ \frac{1230417}{1000000},\ \frac{12788307}{10000000}
	\right).
	\]
	
	\paragraph{Order $r=8$.}
	\[
	(y_0,y_2,y_4,y_6,y_8,y_{10},y_{12},y_{14},y_{16})
	=\left(
	1,\ \frac{49}{48},\ \frac{301}{288},\ \frac{3703}{3456},\ \frac{22799}{20736},\ \frac{561883}{497664},\ \frac{3463859}{2985984},\ \frac{42726971}{35831808},\ \frac{65904559}{53747712}
	\right).
	\]
	
\section{Proof of Theorem \ref{thm:main}}\label{sec:proof}

From Lemma \ref{lem:sos}, for every $r\ge 3$, the polynomials $A_r,B_r$ defined in \eqref{eq:def-Ar-Br}
satisfy the identity \eqref{eq:identity}.
Thus the triple
\[
\varepsilon:=-\frac{1}{r(r-2)},\qquad
p(x):=\frac{A_r(x)^2}{r(r-2)},\qquad
q(x):=\frac{4B_r(x)^2}{r(r-2)}
\]
is feasible for \eqref{eq:SOS}. Therefore
\begin{equation}\label{eq:lower-bound-epsr}
	\sup\eqref{eq:SOS}\ \ge\ -\frac{1}{r(r-2)}.
\end{equation}

Let $\mu$ be the atomic measure \eqref{eq:mu} supported on the roots of $A_r$ and let
$y$ be its moment vector.
By Lemma~\ref{lem:prob-and-obj}, $\ell_y(1)=y_0=1$.
By Lemma~\ref{lem:moment-matrix-psd}, $M_r(y)\succeq 0$.
By Lemma~\ref{lem:psd-localizing}, $M_{r-3}(g\,y)\succeq 0$.
Hence $y$ is feasible for \eqref{eq:MOM}.
Again by Lemma~\ref{lem:prob-and-obj},
\[
\ell_y(f)=-\frac{1}{r(r-2)},
\]
so
\begin{equation}\label{eq:upper-bound-epsr}
	\inf\eqref{eq:MOM}\ \le\ -\frac{1}{r(r-2)}.
\end{equation}

By the weak duality Lemma~\ref{prop:weak-duality}), we have
\[
\inf\eqref{eq:MOM}\ \ge\ \sup\eqref{eq:SOS}.
\]
Combining this inequality with \eqref{eq:lower-bound-epsr} and \eqref{eq:upper-bound-epsr} yields
\[
\inf\eqref{eq:MOM}=\sup\eqref{eq:SOS}=-\frac{1}{r(r-2)}.
\]
By Lemma~\ref{prop:slater}, strong duality and attainment hold on both sides, so this common value
is precisely the relaxation optimum $\varepsilon_r^\star$. This proves Theorem~\ref{thm:main}.
\qed

\section{Conclusion}

In this paper we show that trigonometric properties of orthogonal polynomials can exploited to construct in rational arithmetic an analytic solution of the semidefinite relaxations of Stengle's example \cite{Stengle1996}, thereby settling the question of the exact convergence rate of the moment-SOS hierarchy.

It would be interesting to investigate whether similar techniques could be used to solve analytically the other low-dimensional POP examples used as challenging benchmarks for high-precision semidefinite solvers \cite{Kocvara2025}. 

Key to the analytic solution of problem \eqref{eq:SOS} is the pure square form of the polynomials $p$ and $q$, expressed as a quadratic algebraic equation \eqref{eq:identity} in polynomials $A_r$ and $B_r$. The rank-one structure of the Gram matrices of the SOS polynomials was also exploited in \cite{Henrion2025} to derive tight upper and lower bounds on the values of the moment-SOS hierarchy for a parametric POP problem. It is currently not well understood for which class of POP does optimality of a mom-SOS relaxation imply that the Gram matrices of the SOS dual multipliers are rank-one.
	
	\section*{Acknowledgments}

The author is grateful to Michal Ko\v cvara for sharing the slides of his talk \cite{Kocvara2025} and for insightful discussions. The author acknowledges the use of AI for assistance with brainstorming ideas, mathematical development, coding and drafting the manuscript. The final
content, analysis and conclusions remain the sole responsibility of the author.

\end{document}